\documentclass[psamsfonts, reqno]{amsart}

\usepackage{amscd,amsmath,amssymb,amsfonts,verbatim}
\usepackage{times}
\usepackage[cmtip, all]{xy}
\usepackage{graphicx}
\usepackage[active]{srcltx}
\usepackage{color}
\usepackage{textcomp}

\definecolor{refkey}{gray}{.5}   % graylevel for refs
\definecolor{labelkey}{gray}{.5} % graylevel for labels
\definecolor{Red}{rgb}{1,0,0}

\newcommand{\ra}{\rightarrow}		
\newcommand{\by}[1]{\stackrel{#1}{\ra}}

\newcommand{\surj}{\ra\!\!\!\ra}	
\newcommand{\ol}{\overline}		\newcommand{\wt}{\widetilde}
\newcommand{\iso}{\by \sim}

\newtheorem{theorem}{Theorem}[section]
\newtheorem{proposition}[theorem]{Proposition}
\newtheorem{lemma}[theorem]{Lemma}
\newtheorem{corollary}[theorem]{Corollary}
\newtheorem{question}[theorem]{Question}

\theoremstyle{definition}

\newtheorem{remark}[theorem]{Remark}
\newtheorem{example}[theorem]{Example}
\newtheorem{definition}[theorem]{Definition}

\newcommand{\gv}{\varnothing}

\newcommand{\BC}{\mbox{$\mathbb C$}}	
	\newcommand{\BF}{\mbox{$\mathbb F$}}

	\newcommand{\CR}{\mbox{$\mathcal R$}}

\def\proof{\paragraph{{\bf Proof}}}

\title{On a  question of Moshe Roitman and Euler class of stably free module}
\author{Manoj K. Keshari and Soumi Tikader}
\newcommand{\Addresses}{{
  \bigskip
  \footnotesize

\textsc{Manoj K. Keshari, Department of Mathematics, IIT Bombay
            Mumbai 400076, INDIA}\par\nopagebreak
  \textit{E-mail:} Manoj K. Keshari \texttt{<keshari@math.iitb.ac.in>}

\medskip
  
 \textsc{Soumi Tikader, Department of Mathematics, IIT Bombay 
            Mumbai 400076, INDIA}\par\nopagebreak
  \textit{E-mail:}  Soumi Tikader \texttt{<tikadersoumi@gmail.com>}

  \medskip
 
  }}

\begin{document}
\maketitle
\subjclass 2020 Mathematics Subject Classification:{13C10, 13B25, 19A13}

 ~ {Projective modules, affine algebra, unimodular elements.}

 \begin{abstract}
Let $A$ be a ring of dimension $d$ containing an infinite field $k$, $T_1,\ldots,T_r$ be variables over $A$ and $P$ be a projective $A[T_1,\ldots,T_r]$-module of rank $n$. Assume one of the following conditions hold.
\begin{enumerate}
\item $2n\geq d+3$ and $P$ is extended from $A$.
\item $2n\geq d+2$, $A$ is an affine $\overline {\mathbb F}_p$-algebra and $P$ is extended from $A$.
\item $2n\geq d+3$ and singular locus of $Spec(A)$ is a closed set $V(\mathcal J)$ with ht $\mathcal J\geq d-n+2$.
\end{enumerate}
Assume $Um(P_f)\neq \gv$ for some monic polynomial $f(T_r)\in A[T_1,\ldots,T_r]$. Then $Um(P)\neq \gv$ (see \ref{32}).

\end{abstract}

\section{Introduction}

{\it All rings are commutative noetherian with unity and all projective modules are finitely generated of constant rank.}

Let $A$ be a ring of dimension $d$ and $P$ be a projective $A[T]$-module of rank $n$. We say that $p\in P$ is a unimodular element if there exist $\phi\in Hom(P,A)$ such that $\phi(p)=1$. We write $Um(P)$ for the set of unimodular elements of $P$. When $n>d$, then Plumstead \cite{Pl} proved that $Um(P)\neq \gv$. Further, there are well known examples in the case $n\leq d$ with  $Um(P) = \gv$.  For example, let $d=2m$, $R=\mathbb R[X_0,\ldots,X_{d}]/(X_0^2+\ldots+X_{d}^2-1)$, $\phi:R^{d+1}\surj R$ defined by $e_i\mapsto X_i$
 and $Q=ker (\phi)$. Then $Q$ is stably free $R$-module of rank $d$ and $Um(Q)=\gv$. Hence if $P=Q[T]$, then $Um(P)=\gv$. Thus we need some further conditions in the case $n\leq d$ to get $Um(P)\neq \gv$.

Let $f\in A[T]$ be a monic polynomial and $P_f=P\otimes A[T]_f$. 
Roitman \cite[Lemma 10]{R} proved that if $A$ is a local ring, then $Um(P_f)\neq \gv$ implies $Um(P)\neq \gv$.
He asked whether this result holds for arbitrary ring $A$.  

\begin{question}[Roitman's question]
Let $A$ be a ring of dimension $d$ and $P$ be a projective $A[T]$-module of rank $n\leq d$. 
Let $f\in A[T]$ be a monic polynomial such that $Um(P_f)\neq \gv$. Does this imply $Um(P)\neq \gv$?
\end{question}

If $A$ contains an infinite field $k$, then
an affirmative answer is given in the following cases. 
\begin{enumerate}
\item Bhatwadekar \cite[Proposition 3.3]{B} for $n=2$, arbitrary $d$ and $A$ need not contain any field.

\item Bhatwadekar-Sridharan \cite[Theorem 3.4]{BR} for $n=d$.

 \item Bhatwadekar-Keshari \cite[Theorem 5.3]{BK} for $2n\geq d+3$ when $P$ is extended from $A$. 

\item Bhatwadekar-Keshari \cite[Corollary 5.4]{BK} for $2n\geq d+3$ when $A$ is regular.
\end{enumerate}

%In this paper we exmine the Roitman's question for the ring which may not be regular but with isolated singularity.

We will follow the proof of Bhatwadekar-Keshari \cite[Theorem 5.3]{BK} with suitable modification and prove the  following generalisation of $(4)$ (see \ref{main}).

\begin{theorem}\label{roitman}
Let $A$ be a ring of dimension $d$ containing an infinite field $k$ and $P$ be a projective $A[T]$-module of rank $n$ with $2n\geq d+3$. Assume the singular locus of Spec$(A)$ is a closed set $V(\mathcal J)$ with ht$(\mathcal J)\geq d-n+2$. Let $f\in A[T]$ be a monic polynomial such that $Um(P_f)\neq \gv$. Then $Um(P)\neq \gv$.
\end{theorem}

If $A$ is an affine algebra over an infinite field $k$, then  singular locus of Spec$(A)$ is a closed set.  In particular, if $A$ is an afffine $k$-algebra having an isolated singularity, then ht $\mathcal J=d$ and (\ref{roitman}) is applicable..

The following result is proved in Mandal-Murthy \cite[Theorem 3.2]{MMu} when $L=I^2$. We will essentially follow their proof (see \ref{main2}).

\begin{theorem}\label{mmu}
Let $A$ be an affine $\ol \BF_p$-algebra of dimension $d\geq 2 $ and $I,L$ be ideals of $A$ with $L\subset I^2$. Let $P$ be a projective $A$-module of rank $d$ and $\phi:P\surj I/L$ be a surjection. Then $\phi$ can be lifted to a surjection $\Phi:P\surj I$.
\end{theorem}

When $A$ is an affine algebra over $\ol\BF_p$, then using (\ref{mmu}), we can improve Bhatwadekar-Keshari \cite[Theorem 5.3]{BK} as follows (see \ref{main31}).

\begin{theorem}\label{mmu1}
Let $A$ be an affine $\ol\BF_p$-algebra of dimension $d$ and $P$ be a projective $A[T]$-module of rank $n$ which is extended from $A$ with  $2n\geq d+2$. Let $f\in A[T]$ be a monic polynomial such that $Um(P_f)\neq \gv$. Then $Um(P)\neq \gv$.
\end{theorem}

If $A$ is a smooth affine $\ol \BF_p$-algebra of dimension $d$ and $P$ is a projective $A[T]$-module of rank $n$, then $P$ is extended from $A$ by Popescu \cite{P}. Thus if $2n\geq d+2$ and $Um(P_f)\neq \gv$ for some monic polynomial $f\in A[T]$, then $Um(P)\neq \gv$ by (\ref{mmu1}).

Using above results, we derive the results stated in the abstract (see \ref{32}) and the following results (see \ref{LGP}, \ref{LGP3}, \ref{eP}, \ref{main61}). When $A$ is regular, $(1)$ is due to Das-Sridharan  \cite[Theorem 3.9]{DRS}, $(2)$ is due to Bhatwadekar-Keshari \cite[Theorem 4.13]{BK}, $(3)$ and $(4)$ are due to Bhatwadekar-Sridharan \cite[Theorem 5.4]{BRS} when $r=1$, $P\oplus R=R^{n+1}$ and $2n\geq d+4$.

\begin{theorem}\label{aaa}
Let $R$ be a ring of dimension $d$ containing a field $k$ and $n$ be an integer with
 $2n \geq d+3$. Assume the singular locus of $Spec(R)$ is a closed set $V(\mathcal J)$ with ht$\mathcal J \geq d-n+2$. Then following holds.
\begin{enumerate}
\item Assume $k$ is infinite. The following sequence of Euler class groups is exact, where $\mathfrak m$ runs over all the maximal ideals of $R$.
$$0 \rightarrow E^n(R) \rightarrow E^n(R[T]) \rightarrow \prod_\mathfrak{m} E^n(R_\mathfrak{m}[T]) .$$

\item Assume $k$ is infinite. Let $P$ be a projective $R$-module of rank $n$ and $I$ be an ideal of $R[T]$ of height $n$. Let $\phi:P[T]\surj I/I^2T$ be a surjection. Then $\phi$ can be lifted to a surjection $\Phi:P[T]\surj I$ if and only if $\phi\otimes R_{\mathfrak m}[T]$ can be lifted to a surjection $:P_{\mathfrak m}[T] \surj IR_{\mathfrak m}[T]$ for all maximal ideals $\mathfrak m$ of $A$.

\item  There exist a well defined Euler class map $e:Um_{r,n+r}(R[T]) \to E^n(R[T])$.

\item  Assume $k$ is infinite. Let $P$ be a stably free $R[T]$-module of rank $n$. Then $Um(P)\neq \gv$ if and only if $e(P) = 0 \in E^n(R[T])$.
\end{enumerate}
\end{theorem}

%%%%%%%%%%%%%%%%%%%%%%%%%%%%%%%%%%%%
\section{Preliminaries}

The following result is due to Keshari-Zinna \cite[Proposition 1.1]{KZ}.

\begin{proposition}\label{KZ}
 Let $R$ be a ring and $P$ be a projective $R[T]$-module. Let $J\subset R$ be an ideal such that $P_s$ is extended from $R_s$ for every $s\in J$. Suppose that 
\begin{enumerate}
\item $P/JP$ contains a unimodular element.
\item If $I$ is an ideal of $(R/J)[T]$ of height $ rank(P)-1$, then there exist $\overline \sigma\in Aut((R/J)[T])$ with $\overline \sigma(T)=T$ and a lift $\sigma\in Aut(R[T])$ of $\overline \sigma$ with $\sigma(T)=T$ such that $\overline \sigma(I)$ contains a monic polynomial in $T$.
\item $EL(P/(T,J)P)$ acts transitively on $Um(P/(T,J)P)$.
\item There exists a monic polynomial $f\in R[T]$ such that $Um(P_f)\neq \gv$. 
\end{enumerate}
Then the natural map $Um(P)\surj Um(P/TP)$ is  surjective. In particular, $Um(P/TP)\neq \gv$ implies $Um(P)\neq \gv$.
\end{proposition}

\begin{definition}
Let $A$ be a ring and $P$ be a projective $A$-module. Then $\sigma\in Aut(P)$ is called isotopic to identity if there exist $\Phi(W) \in Aut(P[W])$ such that $\Phi(0)=Id$ and $\Phi(1)=\sigma$. Here $P[W]=P\otimes_A A[W]$.
\end{definition}

The proof of the following result is contained in Quillen \cite[Theorem 1]{Q}.
\begin{lemma}\label{Q}
Let $B$ be a ring and $P$ be a projective $B$-module. Let $a,b\in B$ generate the ideal $B$ and $\sigma \in Aut_{B_{ab}}(P_{ab})$ which is isotopic to identity. Then $\sigma=\tau_a\circ\theta_b$ where $\tau \in Aut_{B_b}(P_b)$ with $\tau=Id$ modulo $Ba$ and $\theta\in Aut_{B_a}(P_a)$ with $\theta=Id$ modulo $Bb$.
\end{lemma}

Next two results are due to Bhatwadekar-Sridharan \cite[Lemma 3.1, Lemma 3.2]{BR}, respectively.

\begin{lemma}\label{3.1}
Let $A$ be a ring containing an infinite field $k$ and $\wt P$ be a projective $A[T]$-module of rank $n$. If $Um(\wt P_f)\neq \gv$ for some monic polynomial $f\in A[T]$, then there exists a surjection $\alpha:\wt P\surj I$ where $I$ is an ideal of $A[T]$ of height $\geq n$ containing a monic polynomial.
\end{lemma}

\begin{lemma}\label{3.2}
Let $R$ be a ring and $Q$ be a projective $R$-module. Let $(\alpha(T),f(T)):Q[T]\oplus R[T] \surj R[T]$ be a surjection with $f(T)$ monic. Let $pr_2:Q[T]\oplus R[T] \surj R[T]$ be the second projection. Then there exists $\sigma(T)\in Aut(Q[T]\oplus R[T])$ which is isotopic to identity and $pr_2\circ\sigma(T)=(\alpha(T),f(T))$.
\end{lemma}

Next result is due to Bhatwadekar-Keshari \cite[Lemma 4.5]{BK}.

\begin{lemma}\label{4.5}
Let $A$ be a ring with dim $A/\mathcal J(A)=r$, where $\mathcal J(A)$ is the Jacobson radical of $A$. Let $I$ and $L$ be ideals of $A[T]$ such that $L\subset I^2$ and $L$ contains a monic polynomial. Let $P'$ be a projective $A[T]$-module of rank $\geq r+1$. Then any surjection $\phi:P'\oplus A[T] \surj I/L$ can be lifted to a surjection $\Phi:P'\oplus A[T]\surj I$ with $\Phi(0,1)$ a monic polynomial.
\end{lemma}

The following result is due to Mandal-Murthy  \cite[Lemma 3.1]{MMu}.

\begin{lemma}\label{3.11}
 Let $A$ be an affine $\ol\BF_p$-algebra of dimension $d \geq 1$ and
 $J$ be a complete intersection ideal of height $d.$ If $J = (a_1,\ldots,a_{d+1})$, then there exist $(\lambda
_1,\ldots,\lambda_{d+1}) \in Um_{d+1}(A)$ such that $\sum _{i=1} ^ {d+1} \lambda_ia_i = 0.$
\end{lemma}

%%%%%%%%%%%%%%%%%%%%%%%%

%\section{Main Theorems}

\section{Roitman's Question : Proof of Theorem \ref{roitman}} 

\begin{theorem}\label{main}
Let $A$ be a ring of dimension $d$ containing an infinite field $k$ and $P$ be a projective $A[T]$-module of rank $n$ with $2n\geq d+3$. Assume that the singular locus of Spec$(A)$ is a closed set $V(\mathcal J)$ with ht$(\mathcal J)\geq d-n+2$. Let $f\in A[T]$ be a monic polynomial such that $Um(P_f)\neq \gv$. Then $Um(P)\neq \gv$.
\end{theorem}
\begin{proof}
Since $Um(P_f)\neq \gv$ and $A$ contains an infinite field $k$, by (\ref{3.1}), we get a surjection $\Phi:P\surj I$, where $I$ is an ideal of $A[T]$ of height $\geq n$ containing a monic polynomial. If ht$\,I>n$, then $I=A[T]$. Hence $Um(P)\neq \gv$ and we are done. So we assume ht$\,I=n$. The map $\Phi$ induces two surjections $\phi: P\surj I/(I^2T)$ and $\overline \Phi: P/TP\surj I(0)$. 
Let $J=(I\cap A)\cap \mathcal J$ and write $R=A_{1+J}$ and $\wt P=P_{1+J}$. 

{\bf Step 1.} We will show that $Um(\wt P)\neq \gv$ by 
verifying the conditions of (\ref{KZ}) for the ideal $JR$ of $R$.
\begin{enumerate}
\item For any $s \in J$, $A_s$ is a regular ring containing an infinite field $k$. Thus by Popescu \cite{P}, $\wt P_s$ is extended from $R_s$.
\item Since $n-1\geq d-n+2$, height of $J$ is $\geq d-n+2$, so dim$(R/J)\leq d-(d-n+2)= n-2$. By Plumstead \cite{Pl}, $\wt P/J\wt P$ has a unimodular element.
\item Since dim$(R/J)\leq n-2$, any ideal of $(R/J)[T]$ of height $\geq n-1$ contains a monic polynomial in $T$.
\item $EL(\wt P/(T,J)\wt P)$ acts transitively on $Um(\wt P/(T,J) \wt P)$, by Bass cancellation  theorem \cite{Ba}.
\item By hypothesis, there exists a monic polynomial $f\in R[T]$ such that $Um(\wt P_f)\neq \gv$.
\end{enumerate}
 
By (\ref{KZ}), $Um(\wt P)\ra Um(\wt P/T\wt P)$ is surjective. 
 Since $\wt P/T\wt P$ is a projective $R$-module of rank $n$ and $J$ is contained in the Jacobson radical of $R$ with dim$(R/J)\leq n-2$, we get $Um(\wt P/T\wt P)\neq \gv$.  Therefore,  $Um(\wt P=P_{1+J})\neq \gv$. 
 
 {\bf Step 2.} Write $P_{1+J}=Q\oplus A_{1+J}[T]$. By (\ref{4.5}), the surjection
 $\phi\otimes A_{1+J}[T] : Q\oplus A_{1+J}[T] \surj I_{1+J}/(I^2T)_{1+J}$ can be lifted to a surjection
 $\Psi=(\psi,h(T)): Q\oplus A_{1+J}[T] \surj I_{1+J} $
 such that $h(T)$ is a monic polynomial in $A_{1+J}[T]$. The surjection $\Psi$ induces the surjection
 $\overline \Psi : Q/TQ\oplus A_{1+J} \surj I(0)_{1+J}$. Hence $\overline \Phi \otimes A_{1+J}=\overline \Psi$. We can find $a\in J$ such that if $b=1+a$, then there exist a projective $A_b[T]$-module $Q_1$ with $h(T)\in A_b[T]$ monic such that
 \begin{enumerate}
 \item $Q_1\otimes A_{1+J}=Q$,
 \item $P_b=Q_1\oplus A_b[T]$,
 \item $\Psi=(\psi,h(T)):P_b\surj I_b$ is a surjection,
 \item $\overline \Phi \otimes A_b= \Psi \otimes A_b[T]/(T)$.
 \end{enumerate}
 Let $pr_2:Q_1\surj A_b[T] \ra A_b[T]$ be the second projection. Since $a\in J$, $I(0)_a=A_a$ and hence $\overline \Phi_a : P_a/TP_a \surj A_a$ is a surjection. Further, since $A_a$ is a regular ring containing a field $k$, by Popescu \cite{P}, $(Q_1)_a$ is extended from $A_{ab}$. Thus there exist a projective $A_{ab}$ module $Q_2$ such that $(Q_1)_a=Q_2[T]$. Thus $P_{ab}=Q_2[T] \oplus A_{ab}[T]$.
 
 Consider two unimodular elements $\Psi_a =(\psi,h(T))_a$ and $(pr_2)_a$ of $P_{ab}^*=(Q_2[T]\oplus A_{ab}[T])^*$. Since $h(T)$ is a monic polynomial in $A_{ab}[T]$, by (\ref{3.2}), $\Psi_a =(\psi,h(T))_a$ and $(pr_2)_a$ are isotopically connected. Since $A_a$ is regular, by Popescu \cite{P}, kernel of $\Psi_a$ is an extended projective module. Therefore, there exist $\Theta \in Aut(Q_2[T]\oplus A_{ab}[T])$ such that $\Theta(0)=Id$  and $\Psi_a\circ \Theta = \Psi_a(0)\otimes A_{ab}[T]=\overline \Phi_{ab} \otimes A_{ab}[T]$. Thus $\Psi_a$ and $\overline \Phi_{ab}\otimes A_{ab}[T]$ are isotopically connected. Therefore, there exist $\Gamma \in Aut(Q_2[T]\oplus A_{ab}[T])$ such that $\Gamma$ is isotopic to identity and $(\overline\Phi_{ab} \otimes A_{ab}[T])\circ \Gamma=(pr_2)_a$.
 
 By (\ref{Q}), $\Gamma=\Omega'_b\circ \Omega_a$ where $\Omega\in Aut(P_b)$ and $\Omega'\in Aut(P_a)$.
 Hence we have surjections $\Delta_1=pr_2\circ \Omega^{-1}:P_b\surj A_b[T]$ and $\Delta_2:P_a\surj A_a[T]$ such that $(\Delta_1)_a=(\Delta_2)_b$. Therefore, they patch up to yield a surjection $\Delta:P\surj A[T]$. This proves the result. $\hfill \square$
\end{proof}

\section{Generalization of Mandal-Murthy : Proof of Theorem \ref{mmu}}

\begin{theorem}\label{main2}
Let $A$ be an affine $\ol \BF_p$-algebra of dimension $d\geq 2 $ and $I,L$ be ideals of $A$ with $L\subset I^2$. Let $P$ be a projective $A$-module of rank $d$ and $\phi:P\surj I/L$ be a surjection. Then $\phi$ can be lifted to a surjection $\Phi:P\surj I$.
\end{theorem}

\begin{proof}
If $R$ is a ring and $Q$ is a projective $R$-module, then $Um(Q)\neq \gv$ if and only if $Um(Q_{red})\neq \gv$, where $Q_{red}=Q\otimes A_{red}$ and $A_{red}=A/nil(A)$ is the reduced ring. Therefore, it is enough to assume that the ring $A$ is reduced. We will give the proof in steps.

{\bf Step 1.} 
If $\phi' : P \ra I $ is a lift of $\phi$, then $\phi'(P)+L = I$. There exist $c \in L$ such that  $c(1-c) \in \phi'(P)$ and  $(\phi'(P),c) = I.$  By Swan Bertini's theorem \cite[Theorem 1.3, 1.4]{S}, there exist  $\psi \in P^*$ such that ht$(\phi'+c\psi )(P)_c = d$ and $(A/(\phi'+c\psi )(P))_c$ is reduced.  Replacing $\phi'$ with $\phi'+c\psi$, we may assume that  ht$(\phi'(P)_c) = d$ and $(A/\phi'(P))_c$ is reduced. If $J = (\phi'(P),1-c)$, then  $\phi'(P) = I\cap J=IJ$. Since $J_c=\phi'(P)_c$ and $(J,c) = A$, thus we get ht$J =$ ht$J_c= d$ and $A/J$ is reduced.  Hence
$J$ is a product of distinct maximal ideals, say $J=\prod_1^n \mathfrak m_i$, where $\mathfrak m_i$'s are maximal ideals of $A.$

The surjection $\phi':P\surj IJ$ induces a surjection $\ol \phi':P/JP \surj J/J^2$. If $\mathfrak p \supset J$ is a prime ideal, then $\ol \phi'_{\mathfrak p}$ lifts to a surjection $P_{\mathfrak p}\surj J_{\mathfrak p}$. 
 Since $A$ is a reduced ring, using prime avoidance lemma, we may assume that $J_p$ is generated by a regular sequence of $d$ elements. Consequently, $J$ is a local complete intersection ideal. 
 Hence $J$ is a product of distinct smooth maximal ideals $\mathfrak m_i.$ If $\mathcal{J} $ is the ideal defining  the singular locus of $A,$  then $J+\mathcal{J} = A.$

By Mohan Kumar-Murthy-Roy \cite[Theorem 3.6]{MMR} for $d\geq 3$ and Murthy \cite[Corollary 3.4]{Mu} for $d=2$,  $J$ is a complete intersection ideal. By Mohan Kumar-Murthy-Roy \cite[Theorem 3.7]{MMR} for $d\geq 3$ and Krishna-Srinivas \cite{KS} for $d=2$, $P$ has a unimodular element, i.e.  $P \simeq P' \oplus A p_d$ with rank $P' = d-1.$ As $\phi'$ induces an isomorphism of $P/JP \simeq J/J^2,$ we have $\phi'(p_1), \cdots ,\phi'(p_{d-1}), \phi'(p_d)$ is a base for $J$ modulo $J^2$ for some  $p_1,\cdots,p_{d-1}\in P'$.

Let $\mathcal{O}$ denotes the image of the map $p_1 \wedge p_2 \ldots \wedge p_d : \wedge^n P^* \ra A$. Then $\mathcal{O}/(J\cap \mathcal{O})  \simeq  A/J.$ Thus $J+\mathcal{O}=A.$  Since $J$ is comaximal with $L,\mathcal{O}$ and $\mathcal{J}$, if $\mathcal T = L\mathcal{O}\mathcal{J}$, then $J+\mathcal T = A.$ 

{\bf Step 2.} 
We will prove the followings.
\begin{enumerate}
\item There exist  $h\in J$ with $h-1 \in \mathcal T$ such that $A/Ah$ is a smooth of dimension $d-1.$ 
\item  If  bar denote going modulo the ideal $(h)$, then $\ol{J}$ is a complete intersection  ideal of height $d-1.$ Also $\phi'$ induces a surjection $\ol {P} \surj \ol {J}.$  
\item $\ol {P}$ is free with basis $\ol p_1,\ldots,\ol p_d.$ 
\end{enumerate}

{\it Case 1.} $d\geq 3$.
Since $J  + \mathcal T = A$, we get $J^2+\mathcal TJ = J.$ Since  $ \phi'(p_1), \cdots ,\phi'(p_d) $ is a basis for $J$ modulo $J^2,$ then there exist $g_1,\ldots,g_d \in \mathcal TJ$ which forms a  basis for $J$ modulo $J^2$  and $\phi'(p_i) - g_i \in J^2.$
% $ (\bar{g_1},\ldots,\bar{g_d})$ is a base for $J$ modulo $J^2,$ where $f'(p_i) - g_i \in J^2.$
Let $c \in J^2 $ be such that $c-1 \in \mathcal T.$ If $l_d= g_d +c \in J,$ then $l_d -1 \in \mathcal T.$ Applying Swan's Bertini \cite{S}, there exist $a \in J^2\mathcal T$ such that if $h = l_d+a$, then $A/Ah$ is smooth outside the singular locus of $A$ and dim $A/Ah =d-1$. Note $h-1= l_d -1+a \in \mathcal T \subset L \subset I,$ hence $h$ is comaximal with the ideal $\mathcal J$ defining the singular locus of $A$. Therefore $A/Ah $ is smooth. This proves $ (1)$.

Since  $J$ is a product of distinct smooth maximal ideals and $\ol{A}$ is smooth, we get $\bar{J}$ is also product of distinct smooth maximal ideal. Therefore, by Mohan Kumar-Murthy-Roy \cite[Theorem 3.6]{MMR}, $\bar{J}$ is a complete intersection of height $d-1 \geq 2.$ Moreover, $\phi' (P) =  I \cap J$ and  $h-1 \in I$ implies $\bar{\phi'}(\bar{P'}) = \bar{J}.$ This proves $ (2)$.

The construction of $h$ gives a  surjection
$f : P' \oplus Ap_d  \surj J/J^2$ defined by  $f= \phi'$ on $P'$ and $f (p_d) = h.$ Since ${\phi'(p_d)} \equiv {h}$ modulo $J^2$, $f$ induces a surjection  $\ol f:\ol P' \surj \ol{J}/\ol J^2$, where bar denote modulo $(h)$.  Since rank $\bar{P'} = d-1 \geq 2$,  by Mandal-Murthy \cite[Theorem 3.2]{MMu},  $\ol {f}$ has a surjective lift $\ol P' \surj \ol J$. 
We have $\ol{J}$ is complete intersection of height $d-1$ and is a product of distinct smooth maximal ideals over $\ol F_p$. Hence by Murthy \cite[Corrollary 2.5]{Mu},  there exist  $\ol y_1, \ldots , \ol y_{d-2} \in \ol {J}$ such that $\ol{A}/(\ol y_1, \ldots , \ol y_{d-2})$ is a smooth complete intersection curve. Let ``tilde" denote going modulo $(\ol y_1, \ldots , \ol y_{d-2})$. By Murthy \cite[Lemma 2.10]{Mu},  pic$(\wt{\ol{A}})=F^1 K_0 (\wt{\ol{A}})$ is a divisible group. Since $\wt {\ol J}$ a principal (invertible) ideal, we get $\wt{\ol{J}} = \wt {\ol K} ^{(d-2)!}$ for  a complete intersection  ideal $\ol K \subset \ol {A}$ with  ht $\ol K$ = $d-2.$
Therefore, we have $\ol{J} = (\ol y_1, \ldots , \ol y_{n-2})+ \ol K ^{(d-2)!}$ and a surjection $\ol{P}' \surj \ol{J}.$ Then invoking the arguments in the proof of Murthy \cite[Theorem 2.2]{Mu}, we get  $(\ol{P}') - (\ol{A}^{d-1}) = - (\ol A/\ol K).$ Since $\ol K$ is a complete intersection, by \cite[Corollary 3.4]{Mu},  we have  $(\ol{P}') - (\ol{A}^{d-1}) =0$. Thus $\ol{P}'$ is stably free $\ol{A}$ module of rank $d-1.$ By Suslin's cancellation theorem \cite{Su}, $\ol{P}'$ is free. Hence, $\bar{P}$ is  free. This proves $ (3)$.
 
{\it Case 2.} $d=2$.
Given that $J$ is a complete intersection ideal of ht $J = 2$. Let $J = (h_1,h_2).$ Since $J+\mathcal T=A$, $(\hat h_1,\hat h_2) \in Um_2(\hat A)$, where $\hat A=A/\mathcal T$. Applying Mohan Kumar-Murthy-Roy \cite[Theorem 2.4]{MMR}, we get $(h_1',h_2') \in Um_2(A)$ such that $\hat h_1=\hat {h_1'}$ and $\hat h_2=\hat{h_2'}.$ Let $\beta \in SL_2(A)$ be such that $(h_1',h_2') \beta = (0,1).$ Replacing $(h_1,h_2)$ by  $(h_1,h_2)\beta$, we may assume $\hat h_2 =1$. Thus $h_2-1 \in \mathcal{T}  \subset L.$ Further, we may assume $h=h_2$ is non-zero divisor by Suslin-Vaserstein \cite[Lemma 9.2]{SuV}. Note that $\ol{P}=P/hP$ is free $\ol{A}$-module with basis  $\ol p_1, \ol p_2$ and $\ol {J} = (\ol {h_1}) = (\phi'(\ol p_1), \phi'(\ol p_2)).$
% where $\phi'(\bar{p_1} = a_1), \phi'(\bar{p_2} = a_2).$

{\bf Step 3.} 
We have $\ol{A}$ is a smooth affine $\ol \BF_p$-algebra of  dimension $d-1$, $\ol{J}$ is complete intersection of height $d-1$ and $\ol P$ is free, so  $\ol{J} = (\phi' (\ol p_1),\ldots,\phi' (\ol p_d)).$ 
By (\ref{3.11}), there exist $(\ol \lambda_1, \ldots, \ol\lambda_d) \in Um_d(\ol A)$ such that $\sum_1^d \ol \lambda_i \phi'(\ol p_i) = 0.$ Since $\sum_1^d \ol \lambda_i \ol p_i \in Um(\ol{P}),$ 
by Mohan Kumar-Murthy-Roy \cite[Theorem 2.4]{MMR}, there exist a $p' \in Um(P)$ with $\bar{p'} = \sum_1^d \ol\lambda_i \ol p_i) \in Um(\ol{P}).$ Thus $P = P'' \oplus Ap'$ and $\phi(p') = ah$ for some $a\in A.$  
%Since $J$ is a smooth ideal of height $J$ = rank $P,$ it is easy to see that $Aa+J = A.$
If $\mathfrak{m}$ is a smooth  maximal ideal containing $(J,a)$, then $ah \in \mathfrak{m}^2$, as  $h \in J \subset \mathfrak{m}.$ By Chinese remainder theorem, a basis of $J/J^2$ induces a basis of  $\mathfrak{m}/\mathfrak{m}^2.$ Since $ah \in J $ is an element of basis of $J/J^2$, but image of $ah$ is zero modulo $\mathfrak m^2$. Thus we get a contradiction and hence
$Aa+J = A.$

 Now $ah = \phi'(p') \in IJ$ and $h-1 \in L\subset I$ gives $a \in I.$ In view of $I/\phi'(P)= I/IJ=A/J,$ we have $I = (\phi'(P), Aa).$
 Define $\Phi:P = P'' \oplus Ap' \surj I $ by letting $\Phi= \phi'$ on $P''$ and $\Phi(p') = a.$ Then $\Phi$ is surjective  lift of $\phi'.$
 This completes the proof.
$\hfill \square$
\end{proof}
%%%%%%%%%%%%%%%%%%

As an application of (\ref{main2}), we prove the following result which extends Bhatwadekar-Keshari \cite[Lemma 4.4]{BK} in case of an affine $\ol \BF_p$-algebra.

\begin{corollary}\label{4.41}
Let $C$ be an affine $\ol \BF_p $-algebra and $A=C_{1+K}$, where $K$ is an ideal of $C$ with dim $A/K=r \geq 2$. Let $P $ be a projective $A$-module of rank $m\geq r$. Let $I$ and $L$ be  ideals of $A$ with $L \subset I^2$. Let $\phi : P\surj I/L$ be a surjection. Then $\phi$ can be lifted to a surjection $\Psi : P \surj I$.
\end{corollary}

\begin{proof}
Follow the proof of Bhatwadekar-Keshari \cite[Lemma 3.4]{BK} where it is proved for general ring in the case $m\geq r+1$. Let $\ol A=A/K$ and let bar denote reduction modulo $K$. Then $\phi$ induces a surjection $\ol \phi : \ol P \surj
\ol I/\ol L$. By (\ref{main2}), $\ol \phi$ can be lifted to a surjection $\ol \Phi: \ol P \surj \ol I$. Let $\ol\Phi$ be a lift of $\ol\Phi$. Since $\ol I=(I+K)/K=I/(I\cap K)$, we have $\Phi(P)+(I\cap K)=I$. We also have $\Phi(P)+L=I$. Since $K$ is contained in Jacobson radical of $A$, any maximal ideal of $A$ containing $\Phi(I)$ contains $I$. Thus $\Phi(P)=I$ and $\Phi$ is a surjective lift of $\phi$. 
$\hfill \square$
\end{proof}

\begin{corollary}\label{4.51}
Let $C$ be an affine $\ol \BF_p $-algebra and $A=C_{1+K}$, where $K$ is an ideal of $C$ with dim $A/K=r \geq 2$. Let $I$ and $L$ be ideals of $A[T]$ such that $L\subset I^2$ and $L$ contains a monic polynomial. Let $P'$ be a projective $A[T]$-module of rank $m\geq r$. Let $\phi:P'\oplus A[T] \surj I/L$ be a surjection. Then we can lift $\phi$ to a surjection $\Phi:P'\oplus A[T] \surj I$ with $\Phi(0,1)$ a monic polynomial.
\end{corollary}

\begin{proof}
Follow the proof of Bhatwadekar-Keshari \cite[Lemma 4.5]{BK} where it is proved for general ring in the case $m\geq r+1$ and use (\ref{4.41}).
$\hfill \square$
\end{proof}
  
%%%%%%%%%%%%%%%%%%%%
  
\section{Roitman's question for $\ol {\mathbb F}_p$-algebras : Proof of Theorem \ref{mmu1}}

\begin{theorem}\label{main31}
Let $A$ be an affine $\ol\BF_p$-algebra of dimension $d$ and $P$ be a projective $A[T]$-module of rank $n$ which is extended from $A$ with  $2n\geq d+2$. Let $f\in A[T]$ be a monic polynomial such that $Um(P_f)\neq \gv$. Then $Um(P)\neq \gv$.
\end{theorem}

\begin{proof}
We may assume $n\leq d$ by Plumstead \cite{Pl}. 
When $d = 2$, then $n=2$ and we are done by Bhatwadekar \cite[Proposition 3.3]{B}. When $d=3$, then $n=3$ and we are done by Bhatwadekar-Sridharan \cite[Theorem 3.4]{BR}. Therefore, we can assume $d\geq 4$ and $n\geq 3$.
The proof is exactly the same as that of Bhatwadekar-Keshari \cite[Theorem 5.3]{BK}, just use (\ref{4.51}) instead of (\ref{4.5}).
$\hfill \square$.
\end{proof}

\section{Proof of results stated in the abstract}

\begin{theorem}\label{32}
Let $A$ be a ring of dimension $d$ containing an infinite field $k$ and $P$ be a projective $A[T_1,\ldots,T_r]$-module of rank $n$. Assume one of the following holds. 
\begin{enumerate}
\item $2n\geq d+3$ and $P$ is extended from $A$.
\item $2n\geq d+2$, $A$ is an affine $\ol\BF_p$-algebra and $P$ is extended from $A$. 
\item $2n\geq d+3$ and the singular locus of Spec$(A)$ is a closed set $V(\mathcal J)$ with ht$(\mathcal J)\geq d-n+2$.
\end{enumerate}
If $Um(P_f)\neq \gv$ for some monic polynomial $f(T_r)\in A[T_1,\ldots,T_r]$, then $Um(P)\neq \gv$.
\end{theorem}

\begin{proof}
{\it Case $1,2$.} If $\ol P=P/(T_1,\ldots,T_{r-1})P$, then $\ol f\in A[T_r]$ is monic and $Um(\ol P_{\ol f})\neq \gv$. By Bhatwadekar-Keshari \cite[Theorem 5.3]{BK} in case $1$ and by (\ref{main31}) in case $2$, $Um(\ol P)\neq \gv$. Since $P$ is extended from $A$, we get $Um(P)\neq \gv$. 

{\it Case $3$.}
The case $r=1$ is (\ref{main}). Assume $r>1$ and write $R=A[T_2,\ldots,T_r]$. Since $Um(P_f)\neq \gv$ implies $Um((P/T_1P)_f)\neq \gv$, by induction on $r$, we get $Um(P/T_1P)\neq \gv$. Further if $S$ is the set of monic polynomials in $k[T_1]$, then $S^{-1}P$ is projective $B[T_2,\ldots,T_r]$-module where $B=S^{-1}A[T_1]$ is of dimension $d$ with height of $\mathcal JB\geq d-n+2$. Thus again by induction on $r$, $Um(S^{-1}P_f)\neq \gv$ implies $Um(S^{-1}P)\neq \gv$. Thus $Um(P_g)\neq \gv$ for some monic $g\in k[T_1]$. Now we will verify the conditions of (\ref{KZ}) to show that the map $Um(P)\to Um(P/T_1P)$ is surjective. In particular, $Um(P)\neq \gv$.
\begin{enumerate}
\item Since $J=\mathcal JR$ is the ideal defining singular locus of $R$, $R_s$ is regular for every $s\in J$. Hence $P_s$ is extended from $R_s$, by Popescu \cite{P}.
\item Since $R[T_1]/(J)=(A/\mathcal J)[T_1,\ldots,T_r]$ and $dim(A/\mathcal J)\leq d-(d-n+2)\leq n-2$, by Bhatwadekar-Roy \cite{BRoy}, we get $Um(P/JP)\neq \gv$.
\item Since $dim(A/\mathcal J)\leq n-2$, if $I$ is an ideal of $R[T_1]/(J)$ of height $n-1$, then by Suslin \cite[Lemma 6.2]{Su1}, there exist $\ol \sigma \in Aut_{(A/\mathcal J)[T_1]}((A/\mathcal J)[T_1,\ldots,T_r])$ of the form $\ol \sigma(T_i)=T_i+T_1^{n_i}$ for $i=2,\ldots,r$ such that $\ol \sigma(I)$ contains a monic polynomial. Clearly, $\ol\sigma$ lifts to an automorphism of $R[T_1]$. 
\item By Lindel \cite{L}, $EL(P/(T_1,J)P)$ acts transitively on $Um(P/(T_1,J)P)$.
\item $Um(P_g)\neq \gv$ for some monic $g\in R[T_1]$.
\end{enumerate}
Now we are done by (\ref{KZ}). 
$\hfill \square$
\end{proof}

\begin{corollary}\label{3.22}
Let $A$ be a ring of dimension $d$ containing an infinite field $k$ of positive charateristic $p>d$ and $P$ be a projective $A[T_1,\ldots,T_r]$-module of rank $n$ which is stably extended from $A$. Assume one of  the following holds.
\begin{enumerate}
\item $2n\geq d+3$. 
\item $2n\geq d+2$  and $A$ is an affine $\ol\BF_p$-algebra.
\end{enumerate}
 If $Um(P_f)\neq \gv$ for some monic polynomial $f(T_r)\in A[T_1,\ldots,T_r]$, then $Um(P)\neq \gv$.
\end{corollary}

\begin{proof}
By Roitman \cite[Corollary 6]{Ro1}, $P$ is extended from $A$. Use (\ref{32}) to complete the proof.
$\hfill \square$
\end{proof}
%%%%%%%%%%%%%%%%%%%%%%%

\begin{remark}\label{rem1}
Let $A$ be an affine $\ol\BF_p$-algebra of dimension $d$ and $P$ be a projective $A[T]$-module of rank $n$ with $2n \geq d+2$.  Assume that the ideal $\mathcal J$ defining the singular locus of Spec$(A)$ has height $\geq d-n+1$. Let $f\in A[T]$ be a monic polynomial such that $Um(P_f)\neq \gv$. Following the proof of (\ref{main}), we get a surjection $\Phi:P\surj I$ where $I$ is a height $n$ ideal of $A[T]$ containing a monic polynomial. Let $J=(I\cap A)\cap \mathcal J$. If we assume that $Um(P_{1+J})\neq \gv$, then following the proof of 
(\ref{main}) and using (\ref{4.51}) instead of (\ref{4.5}), we get $Um(P)\neq \gv$. Therefore, the natural question under above assumptions is whether $Um(P_{1+J})\neq \gv$? $\hfill \square$
\end{remark}

\begin{remark}
If the answer to above question in (\ref{rem1}) is in affirmative in the case $n=d-1\geq 3$, then we get the following corollary:
Let $A$ be a normal affine $\ol\BF_p$-algebra of dimension $d\geq 4$ and $P$ be a projective $A[T]$-module of rank $d-1$. Let $f\in A[T]$ be a monic polynomial such that $Um(P_f)\neq \gv$. Then $Um(P)\neq \gv$, since normality gives height of the ideal $\mathcal J$ defining singular locus of $A$ is $\geq 2$.
$\hfill \square$
\end{remark}

\section{Lifting surjection $A[T]^n \surj I/I^2T$.}

Let $A$ be a ring of dimension $d$ and $n$ be an integer with $2n \geq d+3$.
Bhatwadekar-Sridharan \cite{BRS} defined $n^{th}$ Euler class  group $E^n(A)$ of $A$. An element of $E^n(A)$ is a pair $(I,\omega_I),$ where $I\subset A$ is an ideal  of height $n$ and  $\omega_I:(A/I)^n \surj I/I^2$ is a surjection. It is proved  in \cite[Theorem 4.2]{BRS} that $(I,\omega_I)$ is zero in $E^n(A)$ if and only if $\omega_I$ can be lifted to a surjection $: A^n\surj I$. Mandal-Yang \cite{MY} extended the definition of Euler class group $E^s(A)$ for any $1\leq s\leq d$.

%The following result is due to Bhatwadekar-Sridharan \cite[Lemma 3.5]{BR1} when $A$ is regular. 

\begin{lemma}\label{nori}
Let $A$ be a ring containing a field $k$ and $P$ be a projective $A$-module. Assume the singular locus of $Spec(A)$ is a closed set  $V(\mathcal{J}).$ Let $I \subset A[T]$ be an ideal, $J= I\cap A \cap \mathcal{J}$ and $\ol{\phi} : P[T] \surj I/I^2 T$ be a surjection. Assume $\ol{\phi}\otimes A_{1+J}[T]$ can be lifted to a surjection $\theta :P_{1+J}[T] \surj I_{1+J}$.
Then $\ol{\phi}$ can be lifted to a surjection $\Phi : P[T] \surj I.$
\end{lemma}

\begin{proof}
Follow the proof of Bhatwadekar-Sridharan \cite[Lemma 3.5]{BR1} where it is proved for regular $A$. Here $s\in J,$ implies  $A_{s}$ is a regular ring containing a field. Thus by Popescu \cite{P}, projective $A_{s(1+s)}[T]$-modules are extended from $A_{s(1+s)}$. Rest of the proof is  same as in \cite[Lemma 3.5] {BR1}.
$\hfill \square$
\end{proof}

\begin{lemma}\label{4.6}
Let $A$ be a ring of dimension $d$ and $n$ be an integer with $2n\geq d+3$. Let $I \subset A[T]$ be  an ideal of height $n$ such that $I + \mathcal{J}(A) A[T] = A[T],$ where $\mathcal{J}(A)$ denotes the Jacobson radical of $A$. Assume ht$\mathcal{J}(A) \geq d-n+2.$ Let $P$ be a projective $A$ module of rank $n$ and $\phi: P[T] \surj I/I^2$ be a surjection. If the surjection $\phi \otimes A(T): P(T) \surj IA(T)/I^2A(T)$ can be lifted to a surjection from $P(T)$ to $IA(T),$ then $\phi$ can be lifted to a surjection $\Phi:P[T] \surj I.$ 
\end{lemma}

\begin{proof}
Follow the proof of  Bhatwadekar-Keshari \cite [Lemma 4.6]{BK} where it is proved with the condition $ht \mathcal J(A)\geq n-1$. The similar proof works in our case. 
$\hfill\square$
\end{proof}

\begin{lemma}\label{4.7}
Let $A$ be a ring of dimension $d$ and $n$ be an integer with $2n\geq d+3.$ Let  $L\subset A$ be an ideal of height $\geq  d-n+2$ and $I,I_1 \subset A[T]$ be ideals of height $n.$ Let $P=P_1 \oplus A$ be a projective $A$-module of rank $n$. Assume $J = I \cap A \cap L \subset \mathcal{J}(A),$ where $\mathcal{J} (A)$ denotes the Jacobson radical of $A$ and $I_1 + (J^2T) = A[T].$ Let $\Phi : P[T] \surj I \cap I_1$ and $\Psi:  P[T] \surj I_1$ be two surjection with $\Phi \otimes A[T]/I_1 = \Psi\otimes A[T]/I_1.$ Then we get a surjection $\Lambda : P[T] \surj I $ such that $(\Phi - \Lambda)(P[T]) \subset I^2T.$
\end{lemma}

\begin{proof}
Follow the proof of  Bhatwadekar-Keshari \cite [Lemma 4.7]{BK} where it is proved with the implicit condition $ht \mathcal J(A)\geq n-1$. The same proof works in our case. 
$\hfill\square$
\end{proof}

When $A$ is regular containing a field $k$, the following result is proved in 
Das-Sridharan \cite[Theorem 2.11]{DRS} when $P$ is free and in
Bhatwadekar-Keshari  \cite[Proposition 4.9]{BK} in general case.

\begin{proposition}\label{monic}
Let $A$ be a ring of dimension $d$ containing a field $k$ and $n$ be an integer with $2n\geq d+3$. Assume the singular locus of $Spec(A)$ is a closed set  $V(\mathcal{J})$ with ht$\mathcal{J} \geq d-n+2.$  Let $I \subset A[T]$ be  an ideal of height $n.$ Let $P$ be a projective $A$-module of rank $n$ and $\psi : P[T] \surj I/I^2T$ be a surjection.  Assume there exist a surjection $\Psi':P[T] \otimes A(T) \surj IA(T)$ which is a lift of $\psi \otimes A(T)$.  Then $\psi$ can be lifted to a surjection $\Psi:P[T] \surj I.$
\end{proposition} 
 
 \begin{proof}
Let $J = I \cap A \cap \mathcal{J}$. Then ht $J \geq d-n+2$ and dim $A/J \leq n-2.$
By (\ref{nori}), we may replace $A$ by $A_{1+J}$ and assume that $J  \subset \mathcal{J}(A).$ Since $n >$  dim$A/\mathcal{J}(A),$ we may assume that $P$ has unimodular element i.e. $P \iso  P_1 \oplus A.$

Applying the moving lemma of Das-Keshari \cite[Lemma 3.1]{DK}, the surjection $\psi : P[T] \surj I/I^2T$ can be lifted to a surjection $\theta: P[T] \surj I''$ of $\psi$ such that
\begin{enumerate}
\item $I = I''+ (J^2 T)$, 
\item $I'' = I \cap I',$ where ht $I'=n$ and
\item $I' + (J^2 T) = A[T].$
\end{enumerate}
 
 The surjection $\Theta(=\theta \otimes A(T)) : P[T]\otimes A(T) \surj IA(T) \cap I'A(T)$ satisfies $\Psi'\otimes A(T)/IA(T) =\Theta\otimes A(T)/IA(T).$ Since dim $A(T) = d$,
applying Bhatwadekar-Sridharan \cite[Proposition 3.2]{BRS} to $\Theta$ and $\Psi'$, we get a surjection  $\Phi': P(T) \surj I'A(T)$ such that $\Phi'\otimes A(T)/I'A(T) =\Theta \otimes A(T)/I'A(T).$  Since $I' +\mathcal{J} (A) = A[T]$, where $\mathcal J(A)$ is the Jacobson radical of $A$, applying (\ref{4.6}) to 
the surjections $\phi \otimes A[T]/I' :P[T]/I'P[T] \surj I'/I'^2$ and $\Phi'$, we get a surjection $\Phi: P[T] \surj I'$ which is a lift of $\phi.$ Applying (\ref{4.7}) for the surjections $\phi$ and $\Phi$, we get our desired result.  $\hfill \square$
\end{proof}

%The following subtraction principle is stated in \cite[Corollary 4.11]{BK} for regular $A$. 

\begin{corollary}\label{subtraction}
Let $A$ be a ring of dimension $d$ containing an infinite field $k$ and $n$ be an integer with $2n\geq d+3$. Assume the singular locus of $Spec(A)$ is a closed set  $V(\mathcal{J})$ with ht$\mathcal{J} \geq d-n+2.$  Let $I,I' \subset A[T]$ be  comaximal ideals of height $n.$ Let $P = P_1\oplus A$ be a projective $A$-module of rank $n$. Suppose we have surjections $\Gamma: P[T] \surj I$ and $\Theta : P[T] \surj I \cap I'$ satisfying $\Gamma \otimes A[T]/I = \Theta \otimes A[T]/I.$ Then we have a surjection $\Psi :P[T] \surj I'$ such that $\Psi \otimes A[T]/I' = \Theta \otimes A[T]/I'.$
\end{corollary} 
 
\begin{proof}
Follow the proof of the subtraction principle of Bhatwadekar-Keshari \cite[Corollary 4.11]{BK} where it is proved for regular $A$ and use (\ref{monic}).
\qed 
\end{proof}

% The following result extends Das-Sridharan \cite[Theorem 3.1]{DRS} where it is proved for regular $A$.

\begin{theorem} 
  Let $A$ be a ring of dimension $d$ containing an infinite field $k$ and $n$ be an integer with $2n \geq d+3.$ Assume the singular locus of $Spec(A)$ is a closed set  $V(\mathcal{J})$ with $ht\,\mathcal{J} \geq d-n+2.$ Let $I\subset A[T]$ be an ideal of height $n$ and $\omega_I:(A[T]/I)^n\surj I/I^2$ be a surjection. Then the element  $(I, \omega_I )\in E^n(A[T])$ is zero if and only if $\omega_I$ can be lifted to a surjection $\Psi : A[T]^n \surj I.$ 
\end{theorem}  
 
\begin{proof}
Follow the proof of Das-Sridharan \cite[Theorem 3.1]{DRS} where it is proved for regular $A$ and use (\ref{monic}) instead of \cite[Theorem 2.11]{DRS}. 
$\hfill \square$
\end{proof}

%The following result is due to \cite[Theorem 2.8]{DRS} for regular $A$.

\begin{proposition}\label{LG}
Let $A$ be a ring of dimension $d$ containing an infinite field $k$ and $n$ be an integer with $2n \geq d+3.$ Assume the singular locus of $Spec(A)$ is a closed set  $V(\mathcal{J})$ with $ht \mathcal{J} \geq d-n+2.$ Let $I\subset A[T]$ be an ideal of height $n$ and $\phi : A[T]^n \surj I/I^2T$ be a surjection. Let $N(I;\phi)$ be the set of all $s \in A$ such that $\phi\otimes A_s[T]$ can be lifted to a surjection $\Phi:A_s[T]^n\surj I_s$. Then $N(I;\phi)$ is an ideal of $A.$
\end{proposition}

\proof
Let $J= I \cap A \cap \mathcal{J}$ and $B=A_{1+J}.$ Note $J$ is contained in the Jacobson radical $\mathcal J(B)$ of $B$, $ht\, \mathcal J(B) \geq d- n +2 $ and dim $B/\mathcal J(B) \leq n-1.$  Follow the proof of Das-Sridharan \cite[Theorem 2.8]{DRS} where it is proved for regular $A$ and use (\ref{nori}, \ref{subtraction}) instead of \cite[Lemma 2.6, Proposition 2.2]{DRS}.
\qed

%As a application of (\ref{LG}), following the proof of \cite[Theorem 2.11]{DRS}, we get the following result.

\begin{corollary}\label{LGP2}
Let $A$ be a ring of dimension $d$ containing an infinite field $k$ and $n$ be an integer with $2n \geq d+3.$ Assume the singular locus of $Spec(A)$ is a closed set  $V(\mathcal{J})$ with $ht \mathcal{J} \geq d-n+2.$ Let $I\subset A[T]$ be an ideal of height $n$ and $\phi : A[T]^n \surj I/I^2T$ be a surjection.
Assume $\phi\otimes A_{\mathfrak m}[T]$ can be lifted to a surjection from $A_{\mathfrak m}[T]^n \surj IA_{\mathfrak m}[T]$ for all maximal ideals $\mathfrak m$ of $A$. Then $\phi$ can be lifted to a surjection $\Phi: A[T]^n\surj I$. 
\end{corollary}

\begin{proof}
Following the proof of Das-Sridharan \cite[Theorem 2.11]{DRS} where it is proved for regular $A$ and use  (\ref{LG}).
$\hfill \square$
\end{proof}

%The following local-global principle for Euler class group is proved in \cite[Theorem 3.9]{DRS} for regular $A$.
\section{Local global principle for Euler class group - Proof of Theorem \ref{aaa}(1,2)}

\begin{theorem}\label{LGP}
Let $A$ be a ring of dimension $d$ containing an infinite field $k$ and $n$ be an integer with $2n \geq d+3.$ Assume the singular locus of $Spec(A)$ is a closed set  $V(\mathcal{J})$ with $ht \mathcal{J} \geq d-n+2.$ Then we have the following exact sequence of groups
$$0 \rightarrow E^n(A) \rightarrow E^n(A[T]) \rightarrow \prod_\mathfrak{m} E^n(A_\mathfrak{m}[T]) $$
where the product runs over all maximal ideals $\mathfrak{m}$ of $A.$
\end{theorem}

\begin{proof}
Follow the proof of local global principle of Das-Sridharan \cite[Theorem 3.9]{DRS} where it is proved for regular $A$ and use (\ref{LG}).
\qed
\end{proof}

%The following result generalizes (\ref{LGP2}) and Bhatwadekar-Keshari \cite[Theorem 4.13]{BK}.
%\subsection{Proof of Theorem \ref{aaa} (2)}

\begin{theorem}\label{LGP3}
Let $A$ be a ring of dimension $d$ containing an infinite field $k$ and $n$ be an integer with $2n \geq d+3.$ Assume the singular locus of $Spec(A)$ is a closed set  $V(\mathcal{J})$ with $ht \mathcal{J} \geq d-n+2.$ Let $I\subset A[T]$ be an ideal of height $n$ and $P$ be a projective $A$-module of rank $n$. Let $\phi : P[T] \surj I/I^2T$ be a surjection.
Assume $\phi\otimes A_{\mathfrak m}[T]$ can be lifted to a surjection from $P_{\mathfrak m}[T] \surj IA_{\mathfrak m}[T]$ for all maximal ideals $\mathfrak m$ of $A$. Then $\phi$ can be lifted to a surjection $\Phi: P[T]\surj I$. 
\end{theorem}

\begin{proof}
Let $S$ be the set of all $s\in A$ such that $\phi\otimes A_s[T]$ can be lifted to a surjection $\Phi:P_s[T]\surj I_s$. Our aim is to show that $1\in S$. If $t\in S$ and $a\in A$, then $at\in S$. For every maximal ideal $\mathfrak k$ of $A$, $\phi\otimes A_{\mathfrak m}[T]$ has a surjective lift. Thus there exist $s\in A-\mathfrak m$ such that $P_s$ is free and $s\in S$. Thus we can find $s_1,\ldots,s_r\in S$ such that $P_{s_i}$ is free and $s_1+\ldots+s_r=1$. Therefore, using induction, it is enough to show that if $s,t\in S$ and $P_s$ is free, then $s+t\in S$. Replacing $A$ by $A_{x+y}$, $x$ by $x/(x+y)$ and $y$ by $y/(x+y)$, we may assume $s+t=1$. We will follow the proof of Bhatwadekar-Keshari \cite[Theorem 4.13]{BK} and indicate only the necessary changes. 

In step 1, if $J=I\cap A\cap \mathcal J$, then by (\ref{nori}), we may replace $A$ by $A_{1+J}$ and assume that $J$ is contained in the Jacobson radical of $A$. Note $ht\,J\geq d-n+2$, thus $dim(A/J)\leq n-2$. Rest of the arguments of step 1 is same, just use (\ref{4.7}).
In step 2, just replace \cite[Corollary 4.11]{BK} with the subtraction principle (\ref{subtraction}). In step 3, the arguments are same. In step 4, use (\ref{subtraction}, \ref{nori}, \ref{monic}) to complete the proof.
$\hfill \square$
\end{proof}

\begin{remark}
The following example is due to Bhatwadekar-Mohan Kumar-Srinivas \cite[Example 6.4]{BR1}.
Let $B=\BC[X,Y,Z,W]/(X^5+Y^5+Z^5+W^5]$. Then $B$ has an isolated singularity at the origin. Thus singular locus of $Spec(B)$ is defined by the maximal ideal $\mathcal J=\mathfrak m=(x,y,z,w)$. Here $dim(B)=3=ht\,\mathcal J$. There exist $a\in B-\mathfrak m$, $A=B_a$, an ideal $I\subset A[T]$ of height $3$ 
and a surjection $\phi:A[T]^3 \surj I/I^2T$ which does not has a surjective lift from $A[T]^3\surj I$.
Infact $\phi \otimes A_{\mathfrak m}[T]$ does not has a surjective lift $:A_{\mathfrak m}[T]^3 \surj IA_{\mathfrak m}[T]$. Note that if $\mathfrak n \neq \mathfrak m$ is another maximal ideal of $A$, then 
$\phi \otimes A_{\mathfrak n}[T]$ does has a surjective lift $:A_{\mathfrak n}[T]^3 \surj IA_{\mathfrak n}[T]$, by \cite[Theorem 4.13]{BK} as $A_{\mathfrak n}$ is regular. 
Our result (\ref{LGP3}) shows that lifting $\phi$ locally is precisely the obstruction for lifting $\phi$ globally.
$\hfill \square$
\end{remark}

\begin{remark}
 (Segre class) Let $A$ be a ring of dimension $d$ containing an infinite field $k$ and $n$ be an integer with $2n \geq d+3.$ Assume the singular locus of $Spec(A)$ is a closed set  $V(\mathcal{J})$ with $ht\,\mathcal{J} \geq d-n+2.$ Following the proofs of Das-Keshari \cite[Section 4]{DK}, where they are proved for regular $A$, we
can prove the following results.
\begin{enumerate}
\item Let $I \subset A[T]$
be an ideal such that $\mu(I/I^2) = n$ and $n+{\rm ht} I = {\rm 
dim} A[T]+ 2.$ Let $\omega_I : (A[T]/I)^n \surj I/I^2$ be a surjection. Following \cite[section 4]{DK}, we can define the 
$n^{th}$ Segre class $s^n(I, \omega_I )$ of $(I,\omega_I)$ as an element of $ E^n(A[T])$. 

\item Let $\omega_I : (A[T]/I)^n \surj I/I^2$
be a surjection, where $n\ge {\rm dim} A - {\rm ht} I + 3$. If $s^n (I, \omega_I ) = 0$ in $E^n(A[T])$, then  $\omega_I$ can be lifted to a surjection $\Theta : A[T]^n \surj I.$
\end{enumerate}
\end{remark}

\section{Euler class of stably free module - Proof of Theorem \ref{aaa}(3,4)}

Let $R$ be regular ring of dimension $d$ containing a field $k$ and $P$ be a  stably free $R$-module of rank $n$  with $P\oplus R=R^{n+1}$, where $2n\geq d+3$. Bhatwadekar-Sridharan \cite[Theorem 5.4]{BRS} 
associated an element $e(P)\in E^n(R)$ and proved that $Um(P)\neq \gv$ if and only if $e(P)$ is zero in $E^n(R)$. 
We will extend  this result by relaxing the regularity assumption on $R$ by the condition $ht \mathcal J\geq d-n+2$, where $\mathcal{J}$ is the ideal defining the singular locus of $R$. Further we will extend this result to arbitrary stably free $R$ and $R[T]$-modules of rank $n.$

%The following result extends \cite[Proposition 5.2]{BRS} where it is proved for regular ring $A$.

\begin{proposition} \label{5.2}
Let $R$ be a ring of dimension $d$ containing a field and  $I\subset R[W]$ be an ideal of height $n$ with $J= I(0)$ a proper ideal of $R$. Let $P$ be a projective $R$-module of rank $n$ and $\alpha (W): P[W] \surj I$ be a surjection. Assume the singular locus of $Spec(R)$ is a closed set $V(\mathcal{J})$ and $P/NP$ is free, where $N = (I \cap R \cap \mathcal{J})^2 $.
Let $p_1,\ldots ,p_n$ be elements of $P$ whose reduction modulo $N$ form a basis of $P/NP$ and let $\alpha(0)(p_i) = a_i\in J$. Then there exists  an ideal $K\subset R$ of height $\geq n$ and comaximal with $N$ such that:
\begin{enumerate}
    \item $I\cap KR[W] = (F_1 (W),\ldots,F_n(W)).$
     \item $F_i (0) - F_i(1) \in K^2$.
      \item $\alpha(W)(p_i) - F_i(W) \in I^2.$
      \item $F_i(0) -a_i \in J^2.$
\end{enumerate}
\end{proposition}

\begin{proof}
Follow the proof of Bhatwadekar-Sridharan \cite[Proposition 5.2]{BRS} where it is proved for regular ring $R$. Note that $a\in N\subset \mathcal J$, hence $R_{a}$ is a regular ring containing a field. Hence by Popescu \cite{P}, projective modules over $R_{a(1+a)}$ are extended from $R_{a(1+a)}$. This is the only place where regularity hypothesis was used in \cite{BRS}.
$\hfill\square$
\end{proof}

\subsection{Euler class of stably free $R[T]$-module $P$}
Let $R$ be a ring of dimension $d$ containing a field $k$, $\CR =R[T]$ and $n$ be an integer with $2n \geq d+3.$ Assume the singular locus of $Spec(R)$ is a closed set  $V(\mathcal{J})$ with $ht\,\mathcal{J} \geq d-n+2.$ 
Let $P$ be a stably free $\CR$-module of rank $n$.
We will define the Euler class $e(P)\in E^n(\CR)$ of $P$ and prove that $Um(P)\neq \gv$ if and only if $e(P)=0$ in $E^n(\CR)$.

Let $r \geq 1$ and $Um_{r,n+r} (\CR)$ be the set of all $r \times (n+r)$ matrices $\sigma$ in $M_{r,n+r}(\CR)$ which  has a right inverse, i.e.  there exists $\tau \in M_{n+r,r}(\CR)$ such that $\sigma\circ \tau=Id_r$.
For any $\sigma \in Um_{r,n+r}(\CR),$ we have an exact sequence 
$$ 0 \rightarrow \CR^r \rightarrow^\sigma\, \CR^{n+r} \rightarrow P \rightarrow 0 $$ 
where $\sigma(v) = v\sigma$ for $v \in \CR^{r}$ and $P$ is a stably free $\CR$-module of rank $n.$ Hence every element of $Um_{r,n+r} (\CR)$ corresponds to a stably free $\CR$-module of rank $n$ and conversely, any stably free $\CR$-module $P$ of rank $n$ will give rise to an element of $Um_{r,n+r}(\CR)$ for some $r.$ We will define a map 
$$e:Um_{r,n+r}(\CR) \rightarrow E^n(\CR)$$ which is a natural generalization of the map $Um_{n+1}(\CR) \rightarrow E^n (\CR)$ defined in \cite{BRS}.
Let $\sigma$  be an element of $Um_{r,n+r}(\CR),$ then 
\[
\sigma=
  \begin{bmatrix}
    a_{1,1} & \cdots & a_{1,n+r} \\
    \vdots  &  \cdots & \vdots \\
    a_{r,1} & \cdots & a_{r,n+r} \\
  \end{bmatrix}
\]
Let $e_1,\ldots,e_{n+r}$ be the standard basis of $\CR^{n+r}$ and let 
$$ P = \CR^{n+r}/\left(\sum^{n+r} _{i=1} a_{1,i}e_i,\ldots, \sum _{i=1} ^{n+r} a_{r,i}e_i \right)\CR. $$
Let  $p_1,\ldots,p_{n+r}$  be the images of  $e_1,\ldots,e_{n+r}$ respectively in $P.$ Then 
$$ P = \sum _{i=1} ^{n+r} \CR p_i  ~\mbox{ with relations } ~\sum_{i=1}^{n+r} a_{1,i}p_i = 0 , \ldots ,\sum_{i=1}^{n+r} a_{r,i}p_i = 0 .$$
To the triple $(P,(p_1,\cdots,p_{n+r}),\sigma),$ we associate an element $e(P,(p_1,\cdots,p_{n+r}),\sigma)$ of $E^n(\CR)$ as follows:
Let $\lambda :P \surj J$ be a generic surjection, i.e. $J \subset \CR$ is an ideal of height $n.$ Since $P \oplus \CR^r = \CR^{n+r}$ and dim $\CR/J = d+1-n \leq n-2$, we get $P/JP$ is a free $\CR/J $ module of rank $n$, by Bass \cite{Ba}. Since $J/J^2$ is surjective image of $P/JP,$ $J/J^2$ is generated by $n$ elements.

Let ``bar" denote reduction modulo $J$. By Bass \cite{Ba} , there exist $\Theta \in E_{n+r}(\ol {\CR})$ such that $$[\ol{a_{1,1}} ,\ldots ,\ol{a_{1,n+r}}]\Theta = [1,0,\ldots,0]$$ i.e.
the first row of $\Theta ^{-1}$ is $[\ol{a_{1,1}} ,\ldots ,\ol{a_{1,n+r}}]$. Let $\ol{\sigma}\circ\Theta$ be given by 
\[
\ol \sigma\circ \Theta=
  \begin{bmatrix}
    1       & 0 &0 &0   \\
   \ol{b_{2,1}} & \ol{b_{2,2}} &\cdots & \ol{b_{2,n+r}} \\
   \vdots  &  \vdots &  \vdots & \vdots \\
   \ol{b_{r,1}} & \ol{b_{r,2}} &\cdots & \ol{b_{r,n+r}} 
  \end{bmatrix}.
\]
Note that $[\ol{b_{2,2}} ,\ldots , \ol{b_{2,n+r}}] \in Um_{n+r-1} (\ol {\CR}).$ By Bass \cite{Ba},  there exist $\Theta_1 \in E_{n+r-1}(\ol {\CR})$ such that $$[\ol{b_{2,2}} ,\ldots ,\ol{b_{2,n+r}}]\Theta_1 = [1,0,\ldots,0].$$ If $\Phi \in E_m(\CR)$, then  $
\begin{bmatrix}
Id_t &0 \\
0 &\Phi \\
\end{bmatrix}
\in E_{m+t}(\CR)$.
Let
\[
\ol{\sigma}\circ \Theta\circ \Theta_1 =
 \begin{bmatrix}
    1       & 0 &0 &0   \\
   \ol{b_{2,1}} & 1&\cdots & 0\\
    \ol{b_{3,1}}& \ol{c_{3,2}} &\cdots & \ol{c_{3,n+r}} \\
   \vdots  &  \vdots &  \vdots & \vdots \\
   \ol{b_{r,1}} & \ol{c_{r,2}} &\cdots & \ol{c_{r,n+r}} \\
  \end{bmatrix}
\]
Counting in this way, we get $\tilde{\Theta} \in E_{n+r} (\ol {\CR})$ such that 
\[
\ol{\sigma}\circ \tilde{\Theta} =
 \begin{bmatrix}
    1       & 0 &0 &0 &\cdots &0   \\
   \ol{b_{2,1}} & 1 &0 &0  &\cdots & 0\\
    \ol{b_{3,1}}& \ol{c_{3,2}} &1 &0 &\cdots & 0 \\
   \vdots  &  \vdots &  \vdots & \vdots \\
   \ol{b_{r,1}} & \ol{c_{r,2}}&\cdots & \ol{d_{r,r-1}} &1 &\cdots & 0 \\
  \end{bmatrix}.
\]
We can find elementary matrix $\Psi \in E_{n+r} (\ol { \CR})$ such that 
$$\ol{\sigma}\circ \tilde{\Theta} \circ \Psi = [Id_r,\underline{0}]$$
 where $\underline{0}$ is $r \times n$ zero matrix. 
Let $\Delta = (\tilde{\Theta}\circ \Psi)^{-1} \in E_{n+r}(\ol {\CR}),$ then $\ol{\sigma}$ is the first $r$ rows of $\Delta$, i.e, $\ol{\sigma}$ can be completed to an elementary matrix $\Delta.$ Since
$$  \sum_{i=1}^{n+r} a_{1,i}p_i = 0 , \ldots ,\sum_{i=1}^{n+r} a_{r,i}p_i = 0 .$$ we get $$ \Delta [\ol{p_1},\ldots,\ol{p_{n+r}}]^{t} = [0,\ldots,0,\ol{q_1},\ldots,\ol{q_n}]^t, $$ where $t$ stands for transpose.
Thus $(\ol{q_1},\ldots,\ol{q_n})$ is a basis of the free module $P/JP.$ 

Let $\omega_J : (\CR/J)^n \surj J/J^2$
 be the surjection given by the set of generators of $\ol{\lambda(q_1)},\cdots,\ol{\lambda(q_n)}$ of $J/J^2.$
We define $$e(P,(p_1,\cdots,p_{n+r}),\sigma) = (J,\omega_J) \in E^n(\CR).$$ 
We need to show that $e(P,(p_1,\cdots,p_{n+r}),\sigma)$ is independent of the choice of the elementary completion of $\ol{\sigma}$ and the choice of the generic surjection $\lambda.$

We begin with the following result which shows that  $e(P,(p_1,\cdots,p_{n+r}),\sigma)$ is independent of the choice of the elementary completion $\ol{\sigma}$.

\begin{lemma} \label{2.1}
Suppose $\Gamma \in E_{n+r}(\ol {\CR})$ is chosen so that  its first $r$ rows are $\ol{\sigma}$. Let  $ \Gamma [\ol{p_1},\ldots,\ol{p_{n+r}}]^{t} = [0,\ldots,0,\ol{q_1'},\ldots,\ol{q_n'}]^t.$ Then  there exist $\Psi \in E_n(\ol {\CR})$ such that $\Psi[\ol{q_1},\ldots,\ol{q_n}]^t = [\ol{q_1'},\ldots,\ol{q_n'}].$
\end{lemma}

\begin{proof}
The Matrix $\Gamma \circ \Delta ^{-1} \in E_{n+r}(\ol {\CR})$ is such that its first $r$ rows are $[Id_r,\underline{0}].$ Therefore, there exists $\Psi \in SL_n(\ol {\CR}) \cap E_{n+r}(\ol {\CR})$ such that $\Psi[\ol{q_1},\ldots,\ol{q_n}]^t = [\ol{q_1'},\ldots,\ol{q_n'}].$ 
Since $n > dim \ol {\CR} +1$, by Suslin-Vaserstein \cite{SuV}, $\Psi \in E_n(\ol {\CR}).$
$\hfill\square$
\end{proof}
\medskip

Let $\tilde{\omega_J} : (\CR/J)^n \surj J/J^2$ be the surjection given by the set of generators $\ol{\lambda(q_1')},\cdots,\ol{\lambda(q_n')}$ of $J/J^2.$
Then, by (\ref{2.1}), $(J,\omega_J) = (J,\tilde{\omega_J})\in E^n (\CR).$ Thus for a given surjection $\lambda :P \surj J,$ the element $e(P,(p_1,\cdots,p_{n+r}),\sigma)$ is independent of the choice of the elementary completion of $\ol{\sigma}.$

Now we have to show that $e(P,(p_1,\cdots,p_{n+r}),\sigma)$ is independent of the choice of the generic surjection $\lambda.$ In other words, we have to show that if 
$\lambda' :P \surj J'$ is another generic surjection, where $J'$ is an ideal of $\CR$ of height $n$ and  $\omega_J' : (\CR/J')^n \surj J'/J'^2$ is a surjection obtained as above by completion of $\sigma$ modulo $J'$ to an element of $E_{n+r}(\CR/J'),$ then  $(J,\omega_J) = (J',\omega_J')\in E^n (\CR).$ 

By Bhatwadekar-Sridharan \cite[Lemma 5.1]{BRS}, there exist an ideal $I \subset \CR[W]$ of height $n$ and a surjection $\alpha(W) : P[W] \surj I$ such that $I(0) =J,\alpha(0) = \lambda $ and $I(1) = J', \alpha(1) = \lambda'.$ Let $N =( I \cap \CR \cap \mathcal{J})^2$. 

Note that $\mathcal J\CR$ is the ideal defining singular locus of $\CR$. Since
 ht$N\geq d-n+2 $ and  dim $\CR/N \leq d+1-(d-n+2)\leq n-1 $, by Bass \cite{Ba}, $P/NP$ is free.
 
 Using (\ref{5.2}), rest of the proof of well definedness of $e(P,(p_1,\cdots,p_{n+r}),\sigma)$ is same as in  \cite{BRS}. We denote  the element $e(P,(p_1,\cdots,p_{n+r}),\sigma)$ of  $E^n(\CR)$ by $e(P)$ or $e(\sigma).$
Following the above arguments for a stably free $R$-module $Q$, we get a well defined element $e(Q)$ of $E^n(R)$. 
Therefore, we have proved the following result.

%\subsection{Proof of Theorem \ref{aaa} (3)}

\begin{theorem}\label{eP}
Let $R$ be a ring of dimension $d$ containing a field $k$ and $n$ be an integer with $2n \geq d+3.$ Assume the singular locus of $Spec(R)$ is a closed set  $V(\mathcal{J})$ with $ht\,\mathcal{J} \geq d-n+2.$ 
Then we have  well defined maps 
$$e:Um_{r,n+r}(R) \to E^n(R), \hspace*{.5in}
e:Um_{r,n+r}(R[T]) \to E^n(R[T])$$ 
In particular, given stably free $R$ (resp. $R[T]$) module $Q$ (resp. $P$) of rank $n$, we can associate an element $e(Q) \in E^n(R)$ (resp. $e(P)\in E^n(R[T])$).
\end{theorem}

%\subsection{Proof of Theorem \ref{aaa} (4)}

\begin{theorem} \label{main61}
Let $R$ be a ring of dimension $d$ containing a field $k$ and $n$ be an integer with
 $2n \geq d+3$. Assume
the singular locus of $Spec(R)$ is a closed set $V(\mathcal J)$ with ht$\mathcal J \geq d-n+2$. Let $Q$ (resp. $P$) be stably free $R$ (resp. $R[T]$)-modules of rank $n$. Then 
\begin{enumerate}
\item $Um(Q)\neq \gv$ if and only if $e(Q) = 0 \in E^n(R)$.
\item Assume $k$ is infinite. Then $Um(P)\neq \gv$ if and only if $e(P) = 0 \in E^n(R[T])$.
\end{enumerate}
\end{theorem}

\begin{proof}
(1) Following the proof of Bhatwadekar-Sridharan \cite[Theorem 5.4]{BRS}.

(2) Assume $e(P) = 0$. Then $e(P \otimes R(T)) = 0 $ in $E^n (R(T)).$  By (1), $Um(P\otimes R(T))\neq \gv$. By (\ref{main}), $Um(P)\neq \gv$.

Conversely, assume $Um(P)\neq \gv$. Let $e(P) = (I,\omega_I)$ for some height $n$ ideal $I \subset R[T]$ and $\omega_I : (R[T]/I)^n \surj I/I^2$ a surjection. Since $R$ contains an infinite field, by Bhatwadekar-Sridharan \cite[Lemma 3.2]{BR1}, we may assume that either $I(0) = R$ or $I(0)$ is height $n$ ideal of $R$. 

When $I(0)=R$, $\omega_I$ can be lifted to a surjection $\ol \phi: (R[T]/I)^n\surj I/I^2T$.

When $I(0)$ has height $n$, then using $Um(P/TP)\neq \gv$ and $e(P/TP)=(I(0),\omega_{I(0)})=0$ in $E^n(R)$, by first part of these theorem. Thus by Bhatwadekar-Sridharan \cite[Theorem 4.2]{BRS}, the surjection $\omega_{I(0)}:(R/I(0))^n\surj I(0)/I(0)^2$ can be lifted to a surjection $\phi_1:R^n\surj I(0)$. Patching $\phi_1$ and $\omega_I$, we get a surjection $\ol \phi : (R[T]/I)^n\surj I/I^2T$ which is a lift of $\omega_I$.

Since $Um(P\otimes R(T))\neq \gv$ and $dim\,R(T)=d$, by Bhatwadekar-Sridharan \cite[Theorem 4.2]{BRS}, $\omega_I\otimes R(T)$ can be lifted to a surjection $\Phi: R(T)^n \surj IR(T)$ which is also a lift of $\ol \phi$. By (\ref{monic}), $\ol \phi$ has a surjective lift $\theta:R[T]^n\surj I$ which is also a lift of $\omega_I$. Thus $e(P) = (I,\omega_I)= 0$ in $E^n(R[T])$. 
$\hfill \square$
\end{proof}

\Addresses

\end{document}